%% file: divisibilitywebs.tex
\documentclass{article}

\usepackage {amsmath, amssymb, amscd, amsthm, mathabx, mathrsfs, hyperref, enumerate, url,graphicx}
\usepackage{caption,subcaption,pinlabel}
\usepackage{epstopdf}
\usepackage[text={6.8in,9in}]{geometry}
\usepackage{color}
\usepackage[all]{xy}

\newtheorem {theorem}{Theorem}[section]
\newtheorem {lemma}[theorem]{Lemma}
\newtheorem {proposition}[theorem]{Proposition}
\newtheorem {corollary}[theorem]{Corollary}
\newtheorem {conjecture}[theorem]{Conjecture}
\newtheorem {definition}[theorem]{Definition}
\newtheorem {question}[theorem]{Question}

\numberwithin{equation}{section}

\newcommand{\Z}{\mathbb{Z}}

\renewcommand{\P}{\widehat{P}}
\newcommand{\Q}{\widehat{Q}}

\makeatletter
\let\@fnsymbol\@arabic
\makeatother

\title{Divisibility of Great Webs and Reducible Dehn Surgery}
\author{ Nicholas Zufelt\footnote{Department of Mathematics, The University of Texas, Austin, TX 78712, USA, \url{nzufelt@math.utexas.edu}}}
\date{}

\begin{document}
\maketitle

\abstract{We use the combinatorial techniques of graphs of intersection to study reducible Dehn surgeries on knots in $S^3$.  In particular, in the event that a reducible surgery on a knot $K$ in $S^3$ of slope $r$ produces a manifold with more than two connected summands, we show that $|r|\leq b$, where $b$ denotes the bridge number of $K$.  As a consequence, this possibility is ruled out for knots with $b \leq 5$ and for positive braid closures.}

\input{introduction}

\input{webstructure}

\input{divisibilitylemma}

\input{positivebraids}

\input{corollaries}

\bibliographystyle{amsalpha}
\bibliography{MasterBibliography}
\end{document}

%% file: introduction.tex
\section{Introduction}\label{sec:introduction}
Let $K$ be a nontrivial knot in $S^3$, and denote by $S^3_{r}(K)$ the result of $r$-Dehn surgery on $K$, for $r\in\mathbb{Q} $.  Recall that a 2-sphere in a 3-manifold is \textit{essential} if it is not the boundary of any ball in the manifold. In this paper, we will be concerned about the case when $S^3_{r}(K)$ is a \textit{reducible} manifold, that is, a manifold which contains an essential sphere.   By \cite{pap:gabaifoliations3}, we may assume that in this situation $S^3_{r}(K)$ will decompose as a connected sum and that $r\neq 0$.  The standard example to consider is the reducible surgery on a cable knot $C_{p,q}(K)$ with pattern knot $K$ \cite{pap:gordonsatellite}:
$$S^3_\frac{pq}{1}(C_{p,q}(K)) \cong L(q,p) \# S^3_\frac{p}{q}(K).$$
In fact, the Cabling Conjecture asserts that this is the only possibility.

\begin{conjecture}[Cabling Conjecture \cite{pap:gonz-short}] 
If $K$ is a nontrivial knot in $S^3$ with $S^3_r(K)$ reducible, then $K$ is a cable knot and $r$ is given by the cabling annulus.  
\end{conjecture}

Here torus knots are considered to be cables.  The Cabling Conjecture is known for many classes of knots \cite{pap:EMcablingconjSI}, \cite{pap:hayashi_shimokawa}, \cite{pap:menasco-this}, \cite{pap:mosertorus}, \cite{pap:scharlemann-reducible}, \cite{pap:wuarborescent}, so by \cite{pap:thurston3mfldkleinian} it suffices to assume $K$ is hyperbolic. Additionally it is known that an arbitrary reducible surgery on a knot in $S^3$ coarsely resembles the cabled surgery: at least one summand is a lens space \cite{pap:gordon-luecke-complement}, and the reducing slope $r$ is an integer \cite{pap:gordon-luecke-integral}.  It follows that the $S^3_\frac{p}{q}(K)$ summand of the cable knot surgery is irreducible, so a cable knot's reducible surgery has two irreducible summands. This can be viewed as another approximation to the Cabling Conjecture, and remains unsolved.

\begin{conjecture}[Two Summands Conjecture] 
If $K$ is a nontrivial knot in $S^3$ with $S^3_r(K)$ reducible, then $S^3_r(K)$ consists of two irreducible connected summands.  
\end{conjecture}

In this paper, we study reducible surgeries on knots in $S^3$ in both the general setting and this case of many summands.  Suppose $S^3_r(K)$ is reducible and contains more than two prime summands.  Then \cite{pap:valdezsanchez} shows that all but one of the summands is a lens space, and \cite{pap:howiescottwiegold} shows that all but two of the summands are integral homology spheres.  Hence, $S^3_r(K)\cong L_1\# L_2\# Z$, where each $L_i$ is a lens space and $H_1(Z;\mathbb{Z})=0$. Additionally, there is a history of providing bounds on the surgery coefficient in terms of the bridge number $b$ of $K$.  To start, it is a consequence of the standard combinatorial techniques used that $|r|\leq b(b-1)$.  In \cite{pap:sayaribridge} this bound is improved to $|r|\leq (b-1)(b-2)$, and then to $|r|\leq \frac{1}{4}b(b+2)$ in \cite{pap:howiethreesummands}.  Notice that these bounds are all quadratic in the bridge number.  This is a consequence of the fact that they all arise from bounding $|\pi_1(L_i)|$ linearly in $b$, along with the fact that $|r|=|\pi_1(L_1)|\cdot |\pi_1(L_2)|$.  Our main result is the establishing of a linear bound in this three summands case.

\begin{theorem}\label{thm:threesummandsbound} If Dehn surgery of slope $r$ on a knot $K$ in $S^3$ produces a manifold with more than two, and hence three connected summands, then $|r|\leq b$, where $b$ is the bridge number of $K$. Consequently, the Two Summands Conjecture holds for knots with $b\leq 5$.
\end{theorem} 

As a consequence, we complete the proof of the Two Summands Conjecture for positive braid closures.

\begin{corollary}\label{cor:2sc_positive_braids} Let $K$ be the closure of a positive braid in $S^3$.  Then Dehn surgery on $K$ produces at most two prime connected summands.
\end{corollary}

The proof of Corollary \ref{cor:2sc_positive_braids} relies on the highly nontrivial fact \cite{pap:lidman_sivek_contact_reducible} that hyperbolic positive knots (including the positive braid closures) could only have a single possible reducing slope of $2g-1$.  As we will see,  positive braid closures have a relationship between their genus and their bridge number that not all knots possess.

We remark that the statement regarding knots with $b\leq 5$ in Theorem \ref{thm:threesummandsbound} may already be known to hold in general: the Cabling Conjecture is shown to hold for knots with $b\leq 4$ in \cite{thesis:hoffman}, and in \cite{pap:hoffman_no_SGC} it is claimed that the $b=5$ case has been additionally established in unpublished work.  Another consequence of our work is a reproving of the Cabling Conjecture for knots with $b\leq 3$ (and in a technically quantifiable way, it is ``almost'' established for knots with $b\leq 5$).  See Corollary \ref{cor:primeweb} in Section \ref{sec:prop} for a precise statement.

The techniques of this paper are the so-called graphs of intersection, introduced in \cite{pap:litherland} and \cite{pap:scharlemannfirst} and made famous  with the proof of the knot complement problem \cite{pap:gordon-luecke-complement}.  We combine the application of these techniques to the three summands problem as in \cite{pap:howiethreesummands} with the main technical lemma of \cite{pap:gordon-luecke-complement}, which says that a certain combinatorial object called a \textit{great web} must exist. This is set up in Section \ref{sec:webstructure}, and the proof of Theorem \ref{thm:threesummandsbound} is completed in Section \ref{sec:prop}; there we also state some additional corollaries of a more technical nature which we prove in Section \ref{sec:corollaries}. Section \ref{sec:positivebraids} is devoted to the proof of Corollary \ref{cor:2sc_positive_braids}.

\section*{Acknowledgments}
The author would like to thank Cameron Gordon for his unending assistance and encouragement, Tye Lidman for multiple helpful conversations and for pointing out that I should consider positive braid closures, as well as James Howie for helpful conversation. Additionally, the author is indebted to Cameron Gordon and John Luecke for sharing their proof of a fact which appears here as a part of Proposition \ref{prop:largewebs} with permission, see the discussion before the Proposition for more information.  The author was partially supported by NSF RTG grant DMS-0636643.

%% file: webstructure.tex
\section{Background}\label{sec:webstructure}

We begin by defining a pair of labeled graphs $G_Q$ and $G_P$, and give some basic properties.  The construction that follows first appeared in \cite{pap:howiethreesummands}, but is part of a much more general field of study.  For a survey of the possible situations in which such graphs of intersection can be defined, and for most of the proof of existence of our main combinatorial object, see \cite{pap:Gordon97combinatorialmethods}. 

Since our results apply to both the specific case of obtaining three summands by Dehn surgery on a knot in $S^3$, and the general case of obtaining \textit{any} reducible manifold from a hyperbolic knot in $S^3$, we need to perform both setups, and discuss how they are related.  We will refer to the former as the ``three summands case'', and the latter as the ``general case''.

\subsection{Construction of the graphs $G_Q$ and $G_P$ in the three summands case}

Suppose that $S^3_r(K)\cong L_1\# L_2\# Z$, where $L_i$ is a lens space with $l_i = |\pi_1(L_i)|$, and $Z$ is an integral homology sphere.  For $i=1,2$, let $P_i$ be a planar surface in the exterior of $K$, completing to a reducing sphere $\P_i$ in $S^3_r(K)$ subject to the restrictions:
\begin{enumerate}
\item $\P_i$ separates $L_i$ from $Z$, in the sense that $S^3_r(K)$ cut along $\P_i$ contains two components, with a punctured $L_i$ in one component and a punctured $Z$ in the other, and \item $p_i := |\partial P_i|$ is minimal among spheres with the above property.
\end{enumerate}
Here ``punctured" means the interiors of some finite number of disjoint 3-balls have been removed.  Let $p=p_1+p_2$.  It can be shown using standard techniques that $P_1$ can be chosen to be disjoint from $P_2$.  Let $P = P_1\cup P_2$.  By \cite{pap:gabaifoliations3}, we may find a meridional planar surface $Q$ in the exterior of $K$, completing to a sphere $\Q$ in $S^3$, such that no arc component of $Q\cap P$ is boundary-parallel in $Q$ or in $P$, and so that $\partial Q \cap \partial P$ is minimized.  This allows us to construct a graph $G_Q$ on $\Q$ whose set of (fat) vertices are the disk components of $\Q\ \setminus\ \text{Int}(Q)$ (\textit{i.e.} meridian disks of the filling solid torus) and whose edges are the arc components of $Q\cap P$.  Similarly, we construct graphs $G_P^i$ on $\P_i$, and define $G_P = G_P^1\amalg G_P^2$.  The \textit{faces}, or complementary regions, of $G_Q$ and $G_P$ have boundaries consisting of arcs alternatingly lying on the vertices and the edges of $G_Q$ and $G_P$; these will be called \textit{corners} and \textit{edges}, respectively.  A face will thus be referred to as an $n$-gon if $n$ is the number of edges in its boundary.  Then the choice of $Q$ above translates to the fact that no face of $G_Q$ or $G_P$ is a monogon.  Throughout the paper, any disk of $\Q$ or $\P$ will be assumed to be a union of faces and vertices of $G_Q$ or $G_P$, respectively, and we will not distinguish between such a disk and the graph contained on it (except to be slightly careful in Definition \ref{def:web}).

Label the boundary components of $P$ as $u_1,\ldots, u_p$ so that they occur in order along the boundary torus.  Notice some of these components lie on $P_1$ and the rest lie on $P_2$.  This allows us to assign a pair of labels to each edge of $G_Q$, corresponding to the two vertices of $G_P$ to which the edge is incident.  We will call an oriented edge a $\lambda$\textit{-edge} if it has a label $\lambda$ at its tail.  Since $r$ is an integer, each boundary component of $P$ intersects each boundary component of $Q$ exactly once.  Thus at a single vertex of $G_Q$, one sees these labels occur once in either clockwise or counterclockwise order based on the orientations of $K$ and $Q$; call those vertices \textit{negative} and \textit{positive}, respectively.  Since $\Q$ is separating, precisely $\frac{q}{2}$ vertices of $G_Q$ are positive. Letting $q = |\partial Q|$, label the boundary components of $Q$ as $v_1,\ldots, v_q$, and proceed similarly. Edges can be given a sign as well: an edge of $G_Q$ (respectively, $G_P$) is \textit{positive} if it connects two vertices of $G_Q$ (respectively, $G_P$) of the same sign; otherwise, the edge is \textit{negative}.  By the orientability of all relevant submanifolds, we have the \textit{parity rule}: an arc component of $Q\cap P$ is a positive edge of $G_Q$ if and only if it is a negative edge of $G_P$.  Note that the construction of $Q$ forces $\frac{q}{2}\leq b$ (this is nontrivial, and comes from the use of \textit{thin position} for knots).  We may add elements of $\Z/p\Z$ to labels of $\partial P$ in the obvious way: given a label $\lambda\in\{1,\ldots,p\}$ the label $\lambda + 1$ is either defined or is taken to be $1$; similarly for $\partial Q$.

The technique of graphs of intersection proceeds by deriving combinatorial restrictions on the graphs, which in turn give rise to topological restrictions on the manifolds.  For example, a disk face of a graph whose corners are all $(\lambda, \lambda+1)$ for some label $\lambda$ and whose vertices are all of the same sign is called a \textit{Scharlemann cycle}, and it is well known (see for example \cite{pap:Gordon97combinatorialmethods}) that a Scharlemann cycle face of $G_Q$ gives rise to a lens space summand of $S^3_r(K)$ (in the context we are in, where $\P$ is a union of spheres).  Clearly then no Scharlemann cycle can exist in $G_P$, else there would be a lens space summand in $S^3$.  The highlighted face in Figure \ref{fig:greatweb} is a Scharlemann cycle.

In the three summands case, however, there are two lens space summands of different orders, and hence two different types of Scharlemann cycles may occur on $G_Q$.  To clarify, let $G_i$ be the subgraph of $G_Q$ consisting of all vertices and with edge set given by $G_Q \cap G_P^i$, so that $G_i$ contains only those edges of $G_Q$ which lie on $G_P^i$.  Then we have the following properties, which essentially say that the graph $G_Q$ is well-behaved.

\begin{proposition}[\cite{pap:howiethreesummands}]\label{lem:howiescsarenice} Suppose $\sigma_i$ is a Scharlemann cycle of $G_i$, so that its labels are $x$ and $y$ in $G_Q$.  Then \begin{enumerate}
\item $y=x+1$, \textit{i.e.} there are no edges of $G_{3-i}$ in $\sigma_i$, so that it is a genuine Scharlemann cycle of $G_Q$,
\item $\sigma_i$ contains exactly $l_i$ edges in its boundary, and
\item if $\sigma_i'$ is another Scharlemann cycle on $G_i$, then its corners are also labeled $(x,x+1)$.
\end{enumerate}
\end{proposition}

Define the \textit{length} of a Scharlemann cycle to be the number of edges in its boundary.  Relabel the boundary components of $P$ so that the corners of any length $l_1$ Scharlemann cycle are labeled (1,2).  Let $(x,x+1)$ denote the corners of any length $l_2$ Scharlemann cycle.  Let $$L=\{3,\ldots,x-1,x+2,\ldots,p\},$$ and call $L$ the set of \textit{regular} labels.  The set of \textit{Scharlemann} labels is thus $\{1,2,x,x+1\}$.

\subsection{Construction of $G_P$ and $G_Q$ in the general case}
In the general case, one begins by defining $P$ to be a planar surface completing to a reducing sphere as before, but now subject to global minimality: $|\partial P|$ is minimal among all planar surfaces completing to reducing spheres in $S^3_r(K)$.  Then $Q$ is defined analogously, as are $G_Q$, $G_P$, the labels and notions of positivity and negativity of vertices and edges, the parity rule, etc.  In the general case, there is exactly one pair of labels (say, $1$ and $2$) and one length (say, $l$) for all the Scharlemann cycles of $G_Q$ \cite{pap:gordon-luecke-distance}, so the regular labels are $L=\{3,4,\ldots,p\}$. In fact, it is clear that in the three summands case, if without loss of generality $p_1\leq p_2$, then removing $P_2$ from the discussion and proceeding analogously brings us to the general case.  

\subsection{Great webs}
In \cite{pap:gordon-luecke-complement} a subgraph called a great web is defined and it is shown that sufficiently large great webs contain Scharlemann cycles.  This is integral to the proof of the knot complement problem, as in that context the two surfaces are both spheres in distinct copies of $S^3$, so the existence of a lens space summand in either is impossible.  

\begin{definition}\label{def:web} A \textnormal{great $k$-web} is a subgraph $\Lambda$ of one of the graphs defined above ($G_Q$, say) satisfying the following properties: 
\begin{enumerate}
\item The graph $\Lambda$ lies in a disk $D_\Lambda$ of $\widehat{Q}$ such that every vertex in $D_\Lambda$ is the same sign and is a vertex of $\Lambda$, and 
\item With precisely $k$ total exceptions, every edge incident to any vertex of $\Lambda$ is an edge of $\Lambda$, i.e. has both of its endpoints at a vertex of $\Lambda$.  These exceptional edges are called $\lambda$\textnormal{-ghosts}, where $\lambda$ is the label of the edge at the vertex in $\Lambda$. 
\end{enumerate}
\end{definition}

See Figure \ref{fig:greatweb}. In practice, we do not differentiate between $\Lambda$ and $D_\Lambda$.  Let us say that two surfaces $R$ and $S$ in a knot exterior \textit{intersect essentially} if one can construct graphs $G_R$ and $G_S$ analogous to above which satisfy all the non-triviality requirements described, see \cite{pap:Gordon97combinatorialmethods} for the most general setup.  The following is the main combinatorial result of \cite{pap:gordon-luecke-complement}, where a notion of \textit{representing a type} is defined and used to arrive at the existence of great webs.

\begin{theorem}[Gordon-Luecke \cite{pap:gordon-luecke-complement}]\label{thm:alltypesorweb}
Suppose $Q$ is a connected planar surface, $P$ is a possibly disconnected planar surface, and $Q$ and $P$ intersect essentially.  Let $p=|\partial P|$, and assume that $\Delta > 1-\frac{\chi(P)}{p}$, where $\Delta$ is the geometric intersection number between the slopes given by $\partial Q$ and $\partial P$.  Then either $G_P$ represents all types or $G_Q$ contains a great $(p - \chi(\P))$-web.
\end{theorem}

It should be noted that the statement of this result in \cite{pap:gordon-luecke-complement} does not \textit{a priori} allow $P$ to be disconnected.  This must be carefully checked.  It follows from \cite{pap:parry} and \cite{pap:gordon-luecke-complement} that $G_P$ representing all types implies that the manifold containing $\widehat{Q}$ contains a summand with nontrivial torsion in its first homology.  In our context, this manifold is $S^3$, so we find that $G_Q$ must contain a great $(p-2)$-web in the general case, or a great $(p-4)$-web in the three summands case.  Throughout the remainder of the paper, all great webs will be assumed to be one of these two options, depending on the relevant case, and we will suppress the integer $k$ in the definition.

\begin{lemma} [Gordon-Luecke]\label{lem:webstruc} Let $\Lambda$ be a great web in either case.\begin{enumerate}
\item There is exactly one $\lambda$-ghost for each $\lambda\in L$.  
\item In the general case, $\Lambda$ contains a length $l$ Scharlemann cycle $\sigma$; in the three summands case, $\Lambda$ contains a length $l_1$ Scharlemann cycle $\sigma_1$ and a length $l_2$ Scharlemann cycle $\sigma_2$.\end{enumerate}
\end{lemma}

It turns out that quite a bit more can be said about the structure of $\Lambda$.  For example, the labeled edges traveling along $\partial \Lambda $ are completely understood, and the ghosts occur in (labeled) order, however it involves a bit more work to prove and won't be necessary for us.  To prove the lemma, we need to define the notion of a \textit{great $\lambda$-cycle}, for $\lambda$ a label in $\{1,\ldots,p\}$.  A great $\lambda$-cycle is a disk $\Delta$ of $\Q$ such that all vertices in $\Delta$ are the same sign, and that $\partial \Delta$ may be oriented clockwise or counterclockwise to be a $\lambda$-edged cycle.  In \cite{thesis:hoffman} (see also \cite{pap:hoffman_no_SGC}) it is shown that if $R$ and $S$ intersect essentially, for $\widehat{R}\cong S^2$ in $S^3$ and $\widehat{S}\cong S^2$ in a reducible manifold, then $G_R$ cannot contain a great cycle which is not a Scharlemann cycle. These are called \textit{strict}, or sometimes \textit{new} great cycles.  It is routine to verify that in the three summands case, this result still holds to say that $G_Q$ cannot contain a strict great cycle.

\begin{proof}  Let $\lambda$ be a label in $\{1,\ldots,p\}$, and suppose there is no $\lambda$-ghost.  Then every edge with a label $\lambda$ at a vertex of $\Lambda$ is an edge of $\Lambda$. Begin at a vertex of $\Lambda$, and construct a $\lambda$-edged path.  This may be accomplished by the parity rule, since the edges of $\Lambda$ are certainly positive and so must have a pair of distinct labels, \textit{i.e.} the incoming edge to a vertex doesn't ``use up" the label $\lambda$ at that vertex.  This path eventually gives rise to a great $\lambda$-cycle, and as such must be a Scharlemann cycle.  Hence $\lambda$ is a Scharlemann label.  Thus the remaining $|L|$ labels must all have exactly one ghost.  In the three summands case, suppose without loss of generality that $\lambda=1$, repeat the process for $x$, then apply Proposition \ref{lem:howiescsarenice}.
%Suppose $\lambda, \lambda+1 \in L$.  Then the edge $e_1$ immediately counterclockwise of the $\lambda$-ghost at a vertex $v_1$ is labeled $\lambda +1$; travel counterclockwise along $\partial \Lambda$ beginning at $v_1$. See Figure \ref{fig:ghostproperty}.  If $e_1$ is a $(\lambda+1, \lambda)$ edge, proceed to $e_2$, the next edge counterclockwise along $\partial \Lambda$ with a label $\lambda +1$.  If $e_2$ is a ghost edge, we are done.  Continue in this manner until we either reach the $(\lambda +1)$-ghost, and are done, or until we reach the first edge $e_n$ which is not a $(\lambda+1, \lambda)$ edge.  At $v_{n+1}$ the label $\lambda$ does not occur on $e_n$, but does lie in $\Lambda$ since $v_{n+1}\neq v_1$.  Beginning at $v_{n+1}$ we may construct $\gamma$, a $\lambda$-edged path beginning at $v_{n+1}$.  Since $\lambda\in L$, $\gamma$ must travel through $\text{int}(\Lambda)$ to $v_1$.  Therefore we see that $\gamma$, along with some subset of $\{e_1, \ldots,e_n\}$ bounds a disk $D \subseteq \Lambda$ satisfying the property that every vertex of $D$ has its edge labeled $\lambda+1$ contained in $D$.  Hence $D$ contains a great $(\lambda +1)$-cycle, but $\lambda +1\in L$, a contradiction.  Therefore $e_n$ must be a $(\lambda+1, \lambda)$ edge, and we proceed in this manner to the $(\lambda + 1)$-ghost.
\end{proof}

%% file: divisibilitylemma.tex
\section{Divisibility of Great Webs}\label{sec:prop}

Our main technical result is the following.  Let $v=|V(\Lambda)|$, for $\Lambda$ a great web in either case.

\begin{proposition}\label{prop:webdivisibilty} In the general case, $l$ divides $v$; in the three summands case, $l_i$ divides $v$ for $i=1,2$. 
\end{proposition}

To prove Proposition \ref{prop:webdivisibilty} as well as Corollary \ref{cor:noPSD}, we use the following construction.  Let $\Gamma$ denote the subgraph of $G_P$ consisting of all the vertices of $G_P$ along with those edges of $G_P$ which in $G_Q$ are the edges of $\Lambda$:
\begin{align*}
V(\Gamma) &= V(G_P), \text{ and} \\
E(\Gamma) &= E(\Lambda).
\end{align*}
See Figures \ref{fig:greatweb} and \ref{fig:gamma} for examples of $\Lambda$ and $\Gamma$, respectively.  Note that $E(\Gamma)$ does not contain (those edges of $G_P$ corresponding to) ghost edges of $\Lambda$.  Let the \textit{regular} and \textit{Scharlemann} vertices of $\Gamma$ be those whose label corresponds to regular and Scharlemann labels, respectively.  Then by Lemma \ref{lem:webstruc} the valence of a regular vertex of $\Gamma$ is $v-1$, and the valence of a Scharlemann vertex of $\Gamma$ is $v$.  We will refer to the edges $E(\Gamma)$ as $\Gamma$\textit{-edges}.  It should be noted that no $\Gamma$-edges are adjacent to one another at any vertex when considered as edges of $G_P$; in particular, both labels of each $\Gamma$-edge correspond to vertices of $\Lambda$, and are hence the same sign, so because $\Q$ is separating there must at the very least be an edge with a label corresponding to a vertex of opposite sign between them.  In fact, there are an odd number of edges of $G_P$ between any pair of $\Gamma$-edges.

\begin{figure}[tb]
\begin{center}
\begin{subfigure}{.49\textwidth}
\begin{center}
\includegraphics[width = .75\textwidth]{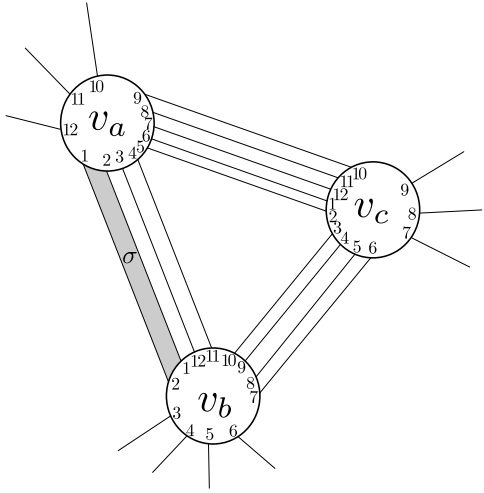}
\caption{A great 10-web $\Lambda$ on $G_Q$ for $p=12$ in the general case}
\label{fig:greatweb}
\end{center}
\end{subfigure}
\begin{subfigure}{.49\textwidth}
\begin{center}
\includegraphics[width = .75\textwidth]{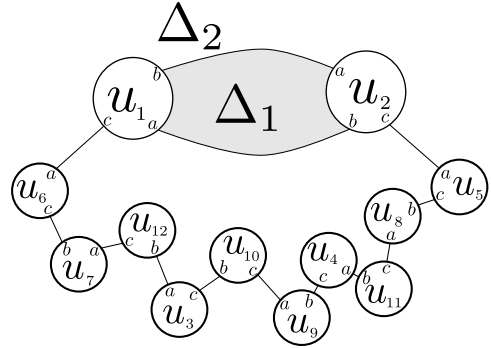}
\caption{The graph $\Gamma$ on $G_P$ for the great web $\Lambda$ pictured in Figure \ref{fig:greatweb}.  Here, $n_1(\Delta_1)=0$, while $n_1(\Delta_2)=1$, because of the location of the $c$-edge at $u_1$.}
\label{fig:gamma}
\end{center}
\end{subfigure}
\caption{The graphs $\Lambda$ and $\Gamma$.  The fact that all the regular vertices of $\Gamma$ are in a single Scharlemann region (namely $\Delta_2$) is a phenomenon of $p$ being quite small, and does not occur in general.}
\end{center}
\end{figure}

\begin{proof}[Proof of Proposition \ref{prop:webdivisibilty}]
Let $E$ denote the $\Gamma$-edges which are incident to two Scharlemann vertices. Thus the edges of $E$ are incident to the vertices $u_1$ and $u_2$ in the general case, and one of the pairs $u_1$ and $u_2$ or $u_x$ and $u_{x+1}$ in the three summands case.  For clarity we will restrict our attention to the edges connecting the vertices $u_1$ and $u_2$, and show that $l$ or $l_1$ divide $v$; the proof of $l_2$ dividing $v$ is the same.  Additionally, for clarity we suppress the difference between the notation in the two cases, namely $\sigma$ and $\sigma_1$, $l$ and $l_1$, and $P$ and $P_1$ and just use the notation from the general case.  Then 
$$\P\ \setminus\ (u_1\cup u_2\cup \bigcup_{e\in E} e)$$ 
consists of disks $D_1,\ldots,D_m$ which we call the \textit{subregions} of $\Gamma$ .  Each disk $D_j$ has a collection of $\Gamma$-edges interior to $D_j$ and incident to $u_1$, call the number of such edges $n_1(D_j)$.  Similarly define $n_2(D_j)$ using $u_2$.  We first claim that $n_1(D_j)=n_2(D_j)$ for each $j$.  To see this, note that all $\Gamma$-edges are negative edges of $G_P$, so $\Gamma$ is a bipartite graph, and hence we may count its edges interior to $D_j$ by counting edges incident to positive vertices or by counting edges incident to negative vertices. Since $u_1$ and $u_2$ are of opposite sign, this gives:
\begin{align*}
n_1(D_j) &\equiv \left|E\big(\Gamma \cap \text{Int}(D_j)\big)\right| \pmod{v-1},\text{ and} \\
n_2(D_j) &\equiv \left|E\big(\Gamma \cap \text{Int}(D_j)\big)\right| \pmod{v-1}.
\end{align*}
Hence we have
$$n_1(D_j) \equiv n_2(D_j) \pmod{v-1}.$$

Since there is at least one Scharlemann cycle ($\sigma$, say) in $\Lambda$ on the labels $\{1,2\}$, we know that $0\leq n_i(D_j) < v-1$ for $i=1,\ 2$.  Hence 
$$n_1(D_j)=n_2(D_j)$$
for each $j\in\{1,\ldots,m\}$.  

Now consider just those edges of $E$ pertaining to $\sigma$; these edges cut $\widehat{P}$ into \textit{Scharlemann regions} $\Delta_1,\ldots, \Delta_l$.  Each $\Delta_k$ is a union of some number of subregions of $\Gamma$ along some edges between $u_1$ and $u_2$.  If $n_1(\Delta_k)$ and $n_2(\Delta_k)$ are defined analogously, we see that $n_1(\Delta_k)= n_2(\Delta_k)$ for each $k$: if $\Delta_k$ contains $m_k$ subregions $D^k_1,D^k_2\ldots, D^k_{m_k}$, then
\begin{align*}
n_1(\Delta_k) &= (m_k-1) + \sum_{j=1}^{m_k} n_1(D^k_j) \\
              &= (m_k-1) + \sum_{j=1}^{m_k} n_2(D^k_j) \\
              &= n_2(\Delta_k).
\end{align*}
 
Note that we could have proven $n_1(\Delta_k)= n_2(\Delta_k)$ for Scharlemann regions from first principals as we did $n_1(D_j)=n_2(D_j)$ for subregions of $\Gamma$.  We provided the proof in this manner because we believe it to be more clear (and slightly stronger).

Consider the torus $T$ obtained by tubing $\P$ along the boundary of the 1-handle $H_1$, where $H_1$ is the portion of the filling solid torus between $u_1$ and $u_2$ not containing fat vertices of $G_P$ in its interior.  On $T$, $\partial \sigma$ is an $\frac{l}{s}$-curve, corresponding to obtaining the lens space summand  $L(l,s)$.  The $\Delta_k$ may be ordered such that the interior $\Gamma$-edges of $\Delta_k$ at $u_1$ run over $H_1$ to the interior $\Gamma$-edges of $\Delta_{k+1}$ at $u_2$, where by $k+1$ we mean $k+1\pmod{l}$.  This follows from the fact that an interior $\Gamma$-edge with a label $\lambda$ at $u_1$ corresponds to an edge labeled $1$ at a vertex $v_\lambda$ in $\Lambda$; since there are no $1$- or $2$-ghosts, the edge labeled $2$ at $v_\lambda$ is an edge of $\Lambda$, and hence the $\lambda$-edge at $u_2$ is a $\Gamma$-edge.  Since the converse is also true, we see that $n_1(\Delta_k)=n_2(\Delta_{k+1})$, and thus $n_i(\Delta_k)=n$ is the same for all $k$ and for $i=1,2$;  this implies that $(1+n)\cdot l = v$ by counting the valence of $u_1$ in $\Gamma$, and thus $l$ divides $v$.
\end{proof}

\begin{proof}[Proof of Theorem \ref{thm:threesummandsbound}]
Recall that $l_1$ and $l_2$ are relatively prime, and that $|r|=l_1\cdot l_2$.  Both of these are facts about $H_1(S^3_r(K))$.  We thus see that $r$ divides $v$ (possibly trivially, \textit{i.e.} $r=\pm v$).  Since $v\leq \frac{q}{2}\leq b$, the result follows.  The final comment is due to the fact that $2 \leq l_1<l_2$ (without loss of generality).
\end{proof}

To conclude this section, we state two corollaries with a more technical flavor to them that additionally result from Proposition \ref{prop:webdivisibilty}.  Their proofs appear in Section \ref{sec:corollaries}.  First, we reprove the Cabling Conjecture for 2- and 3-bridge knots.  This was originally done in \cite{thesis:hoffman}.  This is accomplished by seeing that the number of vertices in a great web cannot be prime, something which is forced by those low bridge numbers.

\begin{corollary}\label{cor:primeweb} In the general case, $l$ cannot equal $v$, so $v$ cannot be prime. Hence the Cabling Conjecture holds for knots with $b\leq 3$, and modulo the case $l=2$, $v=4$, for knots with $b\leq 5$.
\end{corollary}

In \cite{pap:howiethreesummands}, a subgraph called a \textit{sandwiched disk} is defined and used to establish the bound $|r|\leq \frac{1}{4}b(b+2)$ in the three summands case.  A sandwiched disk in $G_Q$ has boundary the union of two arcs, each of which is a subarc of a  Scharlemann cycle, with the additional property that there are no vertices in the interior of the disk.  It is shown that a sandwiched disk  gives rise to a contradiction through a careful analysis of its interior structure.  Since $\Lambda$ needs to contain a Scharlemann cycle of each length $l_i$ and all its vertices are the same sign, it appears to be a good place to search for sandwiched disks.  The following result says that when we restrict our attention to the great web, disks whose boundary is sandwiched cannot exist, independent of the number of vertices in its interior.

\begin{corollary}\label{cor:noPSD}
In the three summands case, no Scharlemann cycle of $\Lambda$ intersects another in more than one vertex.
\end{corollary}

%% file: positivebraids.tex
\section{Positive Braids} \label{sec:positivebraids}
In this section we complete the proof that closures of positive braids satisfy the Two Summands Conjecture.  As a result of Theorem \ref{thm:threesummandsbound}, $|r|\leq b$.  As discussed in the introduction, \cite{pap:lidman_sivek_contact_reducible} established the extremely strong condition that if a hyperbolic positive knot has a reducible Dehn surgery of slope $r$, then $r=2g-1$.  Thus we must relate the bridge number of $K$ to the genus of $K$, a task which may be accomplished if one restricts from positive knots to the closure of positive braids.  Recall that a braid on $n$ strands is built from the braid letters $\{\sigma_i^{\pm 1}\}_{i=1}^{n-1}$, where $\sigma_i$ positively exchanges the $i$th and $(i+1)$st strands.

\begin{proof}[Proof of Corollary \ref{cor:2sc_positive_braids}].  Let $K$ be the closure of a positive braid.  Then $K$ is a positive knot, so by \cite{pap:lidman_sivek_contact_reducible} if $K$ has a reducible Dehn surgery of slope $r$, then $r=2g-1$, where $g$ denotes the genus of $K$.  Hence it suffices to show that if $K$ has a reducible surgery with three summands, then $r\neq 2g-1$.  We may assume by Theorem \ref{thm:threesummandsbound} that the bridge number $b$ of $K$ is at least 6, and that $|r|\leq b$.  Let $\beta$ be a positive braid representing $K$ subject to the constraint that the strand number $n$ of $\beta$ is minimized among positive braids representing $K$.  By \cite{pap:stallings_constructions}, the surface obtained from $\hat{\beta}$ by taking $n$ disks and adding positive bands for each braid letter $\sigma_i$ is fibered, and hence minimal genus.  Thus if $e=e(\beta)$ denotes the total exponent sum, we have 
$2g-1 = e-n.$  Let $e_i=e_i(\beta)$ denote the exponent sum of $\sigma_i$ in $\beta$.  Now if any $e_i=1$, then there exists a sphere passing only through this crossing demonstrating $\hat{\beta}$ as a connected sum.  Since $K$ is hyperbolic, one of the two knots is the unknot, and this allows us to destabilize $\beta$ and obtain a positive braid on fewer strands representing $K$. Thus every $\sigma_i$ occurs at least twice in $\beta$, giving $e\geq 2(n-1)$.  In fact, define $e= 2n-2+s$ for $s$ a non-negative integer.  Noting that $\hat{\beta}$ is a bridge presentation for $K$, we see that
\begin{align*}2g-1 &= e-n\\
&= 2n-2+s - n\\
&= n + (s -2) \\
&\geq b + (s-2)\\
&\geq |r| + (s-2),\end{align*}
so if $s\geq 3$ then the proof is complete.  We proceed with an analysis of possible braid words.  Since $b\geq 6$, we have that $\sigma_1$, $\sigma_2$, $\sigma_3$, $\sigma_{n-2}$, and $\sigma_{n-1}$ all exist and are distinct.  Denote braid equivalence by $\sim$.

\begin{figure}[bt]
\begin{center}
\begin{subfigure}{.49\textwidth}
\begin{center}
\includegraphics[width = .3\textwidth]{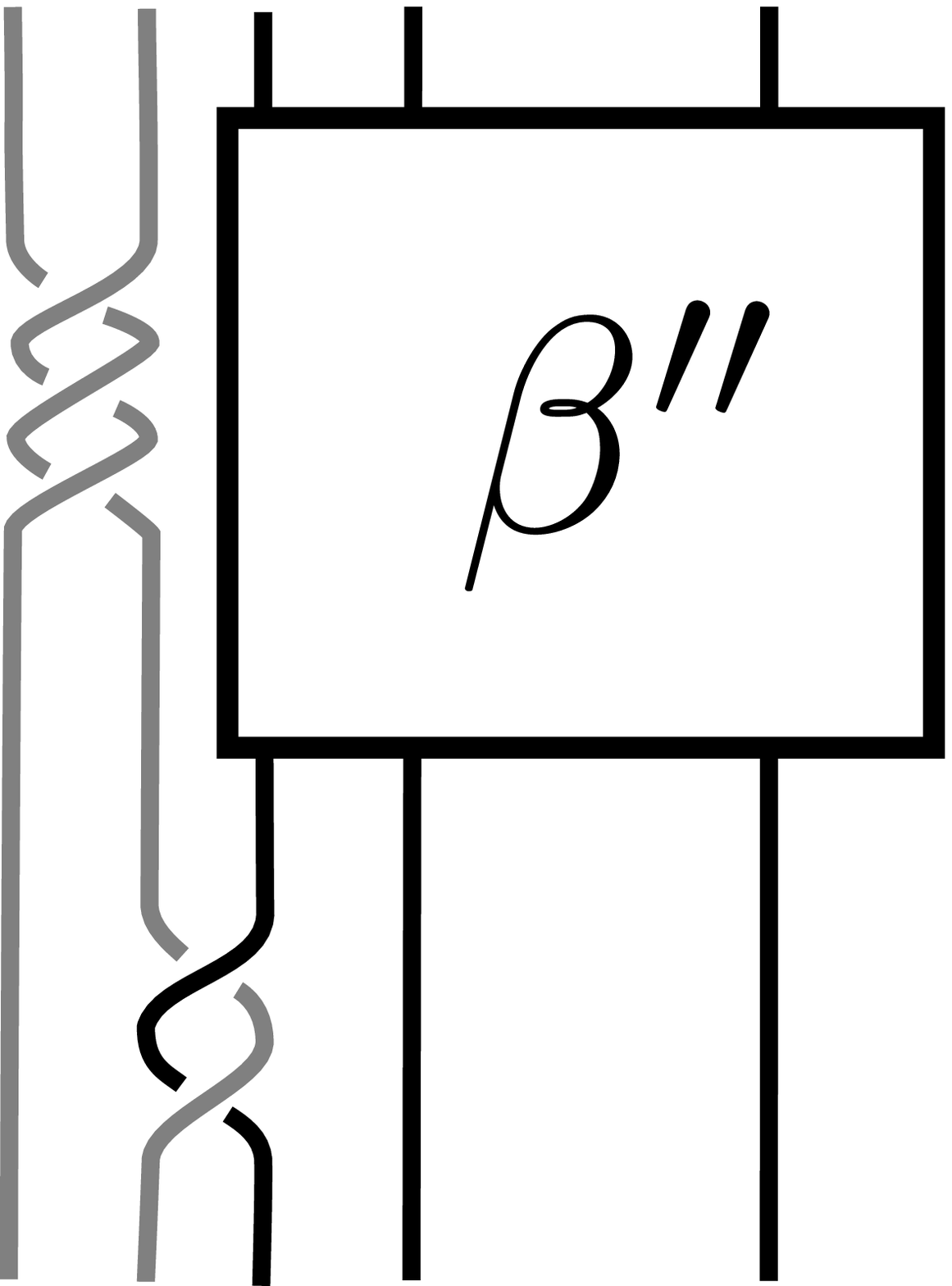}
\caption{The closure of ${\sigma_1}^3\beta''{\sigma_2}^2$ is a link.}
\label{fig:braidmiddlecase}
\end{center}
\end{subfigure}
\begin{subfigure}{.49\textwidth}
\begin{center}
\includegraphics[width = .3\textwidth]{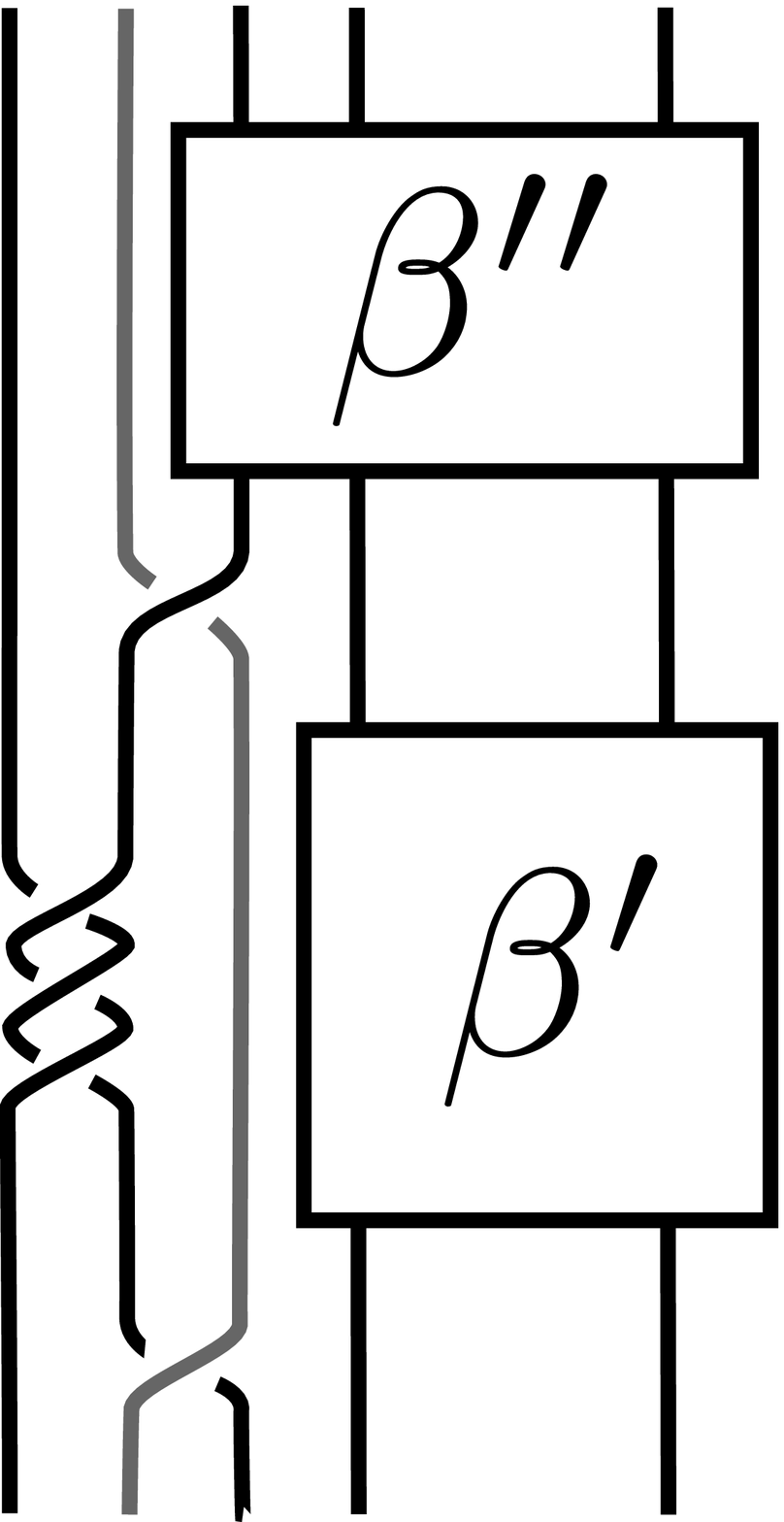}
\caption{The closure of $\beta''\sigma_2{\sigma_1}^3\beta'\sigma_2$ is a link.}
\label{fig:braidfinalcase}
\end{center}
\end{subfigure}
\caption{}
\end{center}
\end{figure}

Our first claim is that at least one of $e_1\geq 3$ or $e_2\geq 4$ must hold.  Suppose $e_1 = 2$ and $e_2\leq 3$, and note that $\beta\not\sim{\sigma_1}^2\beta'$, as then it is clear that $K$ is a link of at least two components.  Thus it must be that 
$$\beta \sim \sigma_1 \beta'\sigma_1\beta'',$$
where $e_2(\beta')$ and $e_2(\beta'')$ are nonzero, as otherwise ${\sigma_1}^2$ could be chosen to occur in $\beta$.  Note $e_1(\beta')=0$ and $e_1(\beta'')=0$.  Since $e_2 \leq 3$, at least one of $\beta'$ and $\beta''$ would have exactly one occurrence of $\sigma_2$.  This means that we could change $\beta$ so that 
\begin{equation}\label{slideandcancel}\begin{aligned}
\beta &\sim \sigma_1\sigma_2\sigma_1\beta'''\\
&\sim \sigma_2\sigma_1\sigma_2\beta'''\\
&\sim \sigma_1\sigma_2\beta'''\sigma_2,\\
&\sim \sigma_2\beta'''\sigma_2,
\end{aligned}
\end{equation}
and $e_1(\sigma_2\beta'''\sigma_2)=0$, a contradiction.  Thus one of $e_1\geq 3$ or $e_2\geq 4$ must hold, and similarly must one of $e_{n-1}\geq 3$ or $e_{n-2}\geq 4$.  If $s\leq 2$, this forces $e_1=3$, $e_{n-1} = 3$, and $e_i=2$ for $i\in\{2,3,\ldots,n-2\}$.

Next we see that $\beta\not\sim {\sigma_2}^2\beta'\ (\sim {\sigma_1}^3\beta''{\sigma_2}^2)$, since $e_2=2$ and hence $K$ is a link of at least two components, see Figure \ref{fig:braidmiddlecase}.  Thus we see that 
$$\beta \sim \sigma_2 \beta'\sigma_2\beta'',$$ 
where $\beta'$ and $\beta''$ each contain a $\sigma_1$ or $\sigma_3$.  Without loss of generality, we have two cases: $e_1(\beta')=2$ or $3$.  If $e_1(\beta')=2$, then we use the relation $\sigma_1\sigma_2\sigma_1= \sigma_2\sigma_1\sigma_2$ twice to remove $\sigma_1$ from $\beta$ as in the process described in \ref{slideandcancel}, a contradiction.  Thus as a last case, assume $e_1(\beta')=3$.  Then either $e_3(\beta')=0$ or $1$, since $e_3(\beta'')\geq 1$.  If $e_3(\beta')=1$, we proceed by using $\beta''$ to reduce $e_3(\beta)$ to $1$ and again see $\hat{\beta}$ as a connected sum, to remove the first three or last $n-4$ strands of $\beta$.  Finally, if $e_3(\beta')=0$, then Figure \ref{fig:braidfinalcase} shows that $K$ is in fact a link of at least two components, a contradiction.

Thus $s\geq 3$, and $K$ could not produce three connected summands by Dehn surgery.
\end{proof}

While the proof of Corollary \ref{cor:2sc_positive_braids} is certainly sufficient, it is perhaps a first approximation to the phenomenon that minimal positive braid words representing a nontrivial knot with sufficiently large bridge number seem to have large exponent sum relative to their strand number (\textit{i.e.} $e> 2n$).  

\begin{question}
For which classes of knots is there a relationship between the bridge number of a knot and the difference $e-n$ of a minimal-strand braid representing the knot?
\end{question}

Note that by the resolution of the Jones Conjecture \cite{pap:dynnikov_prasolov_Jones_Conjecture}, $e-n$ is a knot invariant, where one is careful in defining $e$ to be the \textit{algebraic} exponent sum.  For positive braids, the total exponent sum and the algebraic exponent sum are the same.  The statement $e-n > b$ is false in general, even among positive braids: ${\sigma_1}^3$ is a minimal braid for the trefoil.

It would be interesting to apply a similar argument as in the proof of Corollary \ref{cor:2sc_positive_braids}, if possible, to a class of knots called $L$\textit{-space knots}.  An $L$-space knot is a knot in $S^3$ with a positive surgery resulting in an $L$-space, a manifold whose Heegaard Floer homology is minimal in a suitable sense.  Lens spaces, elliptic manifolds, and connected sums thereof are all $L$-spaces, and there are many others.  See \cite{pap:ozsz3-manifolds_1} for the definition of Heegaard Floer homology, and \cite{pap:ozszhflens} for the definitions and basic facts about $L$-spaces and $L$-space knots.

The argument above required positivity of braids.  There are several generalizations of positivity, such as \textit{homogeneity} \cite{pap:stallings_constructions}, which still gives a fibered knot.  $L$-space knots need not have a positive (or homogeneous) braid representative, but they have a \textit{quasipositive} braid representative \cite{pap:hedden_notions_QP}, which is a different generalization of positivity, and are fibered \cite{pap:nihfkfibered}.  Additionally, they have the same property that the only possible reducing slope among hyperbolic representatives is $2g-1$ \cite{pap:hom_lidman_zufelt}.  Thus they are another good candidate class of knots for which to attempt to resolve the Two Summands Conjecture using Theorem \ref{thm:threesummandsbound} and the techniques of Corollary \ref{cor:2sc_positive_braids}.

%% file: corollaries.tex
\section{Further Analysis of Great Webs}\label{sec:corollaries}

In order to prove Corollary \ref{cor:primeweb} we need to show that, roughly, great webs tend to contain a great deal of topological and algebraic topological information.

\begin{definition} Let $\Lambda$ be a great web in the either the general or the three summands case.  Then $\Lambda$ is \textnormal{small} if there exists a collection of $\frac{p}{2}$ parallel edges in $\Lambda$ containing no Scharlemann cycles, otherwise $\Lambda$ is \textnormal{large}.  Such a collection of parallel edges is called a \textnormal{full quota}.
\end{definition}

The following Proposition uses a pair of techniques which are very different from each other.  The $p\equiv 0\pmod{4}$ case is an adaptation of an argument of Howie \cite[Lemma 2.4]{pap:howiethreesummands}.  The $p\equiv 2\pmod{4}$ case is due to Cameron Gordon and John Luecke; it is an adaptation of an argument of Thompson \cite{pap:thompson_bandsum_propP}.  %The author is indebted to Professors Gordon and Luecke for their generosity in sharing their knowledge.

\begin{proposition}  \label{prop:largewebs}Let $\Lambda$ be a great web in either case.  Then $\Lambda$ is large.
\end{proposition}

\begin{proof}
It is an immediate consequence of \cite[Lemma 2.4]{pap:howiethreesummands} that $\Lambda$ must be large in the three summands case.  In fact, something stronger holds: the bigons given by a full quota not need be located entirely in a single parallelism class of edges.  We will need this rigidity in the general case.

Suppose first that $\Lambda$ is a great web in the general case and $p\equiv 0 \pmod{4}$.  Define the 1-handle $H_i$ to be the component of the filling solid torus contained between the fat vertices $u_i$ and $u_{i+1}$ which contains no other $u_j$.  Let $z\in \P$ and define $g_i$ to be an element of $\pi_1(S^3_r(K),z)$ represented by a curve which travels from $z$ through $\P$ to the vertex $u_i$, over $\partial H_i$ running parallel to the arcs of $Q\cap H_i$ to $u_{i+1}$, then back to $z$ through $\P$.  Note that $\mu = g_1g_2\cdots g_p$ is represented by a meridian of $K$.  Thus $\pi_1(S^3_r(K))/\left<\left<\mu\right>\right>$ is trivial.  However, if $\Lambda$ is small, then $\Lambda$ contains a full quota, which in turn contains bigons $B_i$ which give rise to relations $g_{a+i}g_{a-i}=1$, for $i=1,\ldots,\frac{p-2}{2}$.  To see this relation (without the use of conjugates), add a pair of triangles $\Delta_{a-i}^{a+i+1}$ and $\Delta_{a+i}^{a-i+1}$ in $\P$ to $B_i$, see Figure \ref{fig:pionerelations}.  Thus we see that $\mu=g_awg_bw^{-1}$, for $b = a+\frac{p}{2}$.  

\begin{figure}[hbt]
\begin{center}
\includegraphics[width = .4\textwidth]{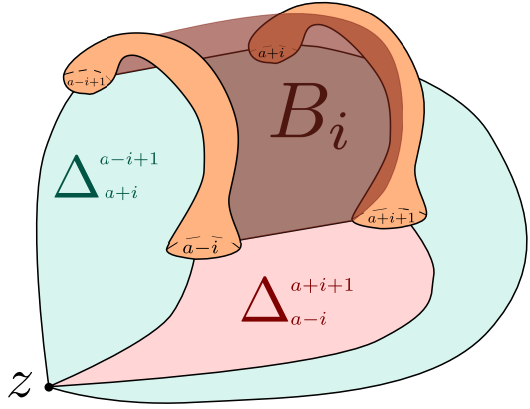}
\caption{The disks extending $B_i$ to give the relation $g_{a+i}g_{a-i}=1$.}
\label{fig:pionerelations}
\end{center}
\end{figure}

Note that $g_a$ and $g_b$ are elements of a factor of $\pi_1(S^3_r(K))=\Z/l\Z * G$.  Since $p\equiv 0\pmod{4}$, these are elements of the same factor, and hence the normal closure $\left<\left<g_awg_bw^{-1}\right>\right>$ has nontrivial index in $\pi_1(S^3_r(K))$.  This is a contradiction.

Hence we may suppose that we are in the general case with $p\equiv 2 \pmod{4}$.  Let $K'$ denote the dual knot of $K$ in $S^3_r(K)$.  Then if $\Lambda$ is small, there is a disk $B$ on $Q$ giving the parallelism  between the outermost edges of the full quota.  This $B$ is a band demonstrating $K'$ as a band sum in $S^3_r(K)$, see Figure \ref{fig:findingbandmeridian}.  Note that the assumption $p\equiv 2 \pmod{4}$ is important, because the tangles (corresponding to cores of $H_a$ and $H_b$) must lie on opposite sides of $\P$ in order for us to say that they do not link.  Let $J$ be a \textit{meridian} of the band $B$, \textit{i.e.} perform the following construction: fix an identification $N(B)\cong B\times [-1,1]$ so that $B_0=B\times \{0\}$ is a slightly wider band than $B$, let $\alpha$ be a properly embedded arc in $B_0$ connecting the two components of $B_0 \cap \partial E(K)$ where $E(K)$ denotes the exterior of $K$ (the `long' edges of the band $B_0$), and define 
$$J = \left(\alpha \times \{\pm 1\}\right)\cup \left(\partial\alpha \times [-1,1]\right).$$
Let $Y$ be the result of performing $0$-framed Dehn surgery on $J$.  Then since $J$ was an unknot in $S^3_r(K)$ (for example, it bounded the disk $\alpha \times [-1,1]$), $Y$ contains an $S^1\times S^2$ connected summand.  In fact, since $J'$, the dual knot to $J$ in $Y$, bounds a disk disjoint from $K'$, we may use this disk to isotope the band $B$ and $K'$ in $Y$ to see that it is an unknotted band which doesn't link either summand of the band sum.  See Figure \ref{fig:unknotting}. That is, in $Y$, $K'$ is a connected sum, and lies in a ball disjoint from the $S^1\times S^2$ connected summand.  Perform the dual surgery on $K'$ to obtain a knot in $S^3$ (that is, $J$) with a surgery containing an $S^1\times S^2$ connected summand.

\begin{figure}[hbt]
\begin{center}
\begin{subfigure}{.49\textwidth}
\begin{center}
\includegraphics[width = .75\textwidth]{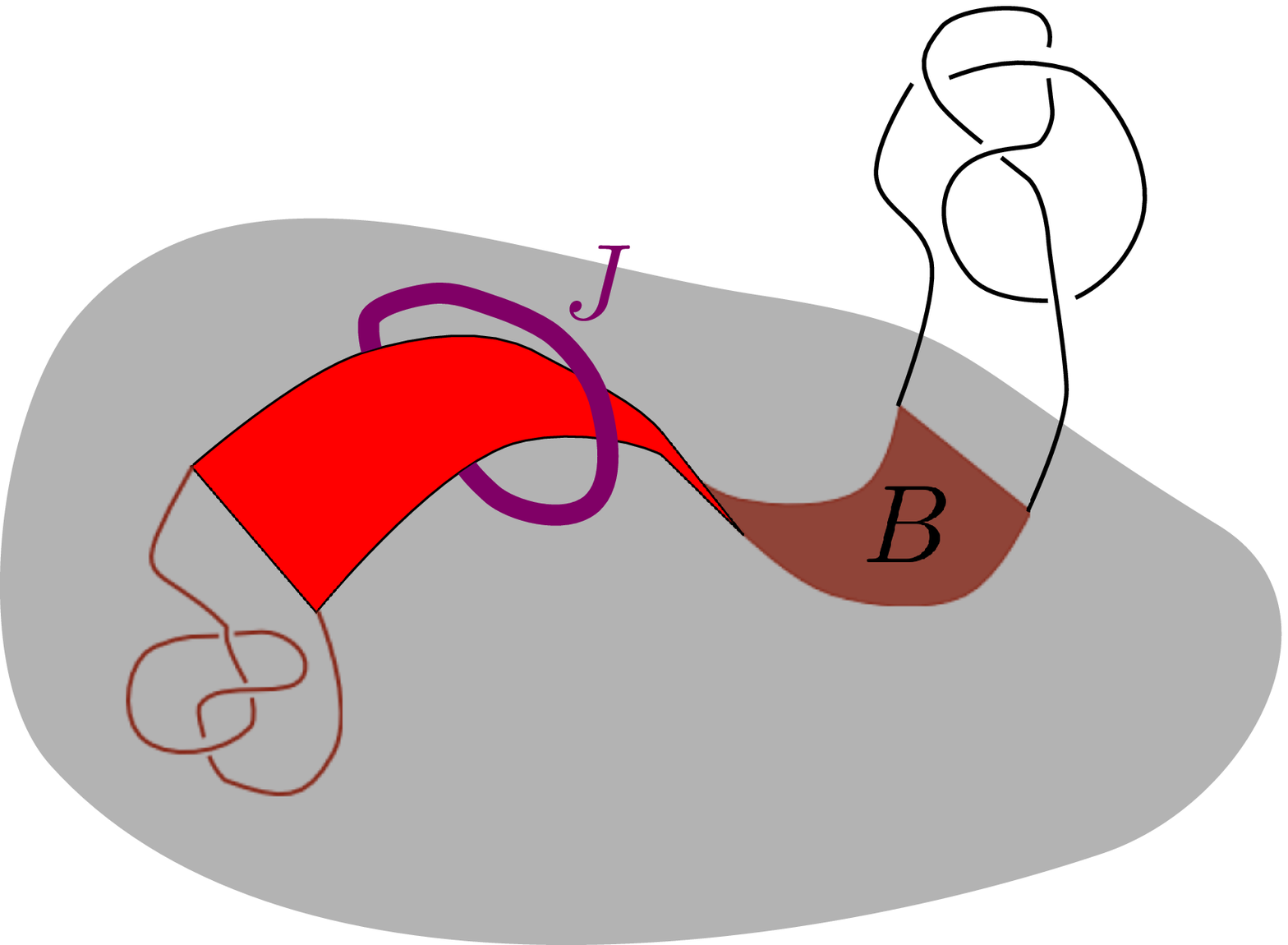}
\caption{$K'$ is a band sum in $S^3_r(K)$, and $J$ is a meridian of the band $B$.}
\label{fig:findingbandmeridian}
\end{center}
\end{subfigure}
\begin{subfigure}{.49\textwidth}
\begin{center}
\includegraphics[width = .75\textwidth]{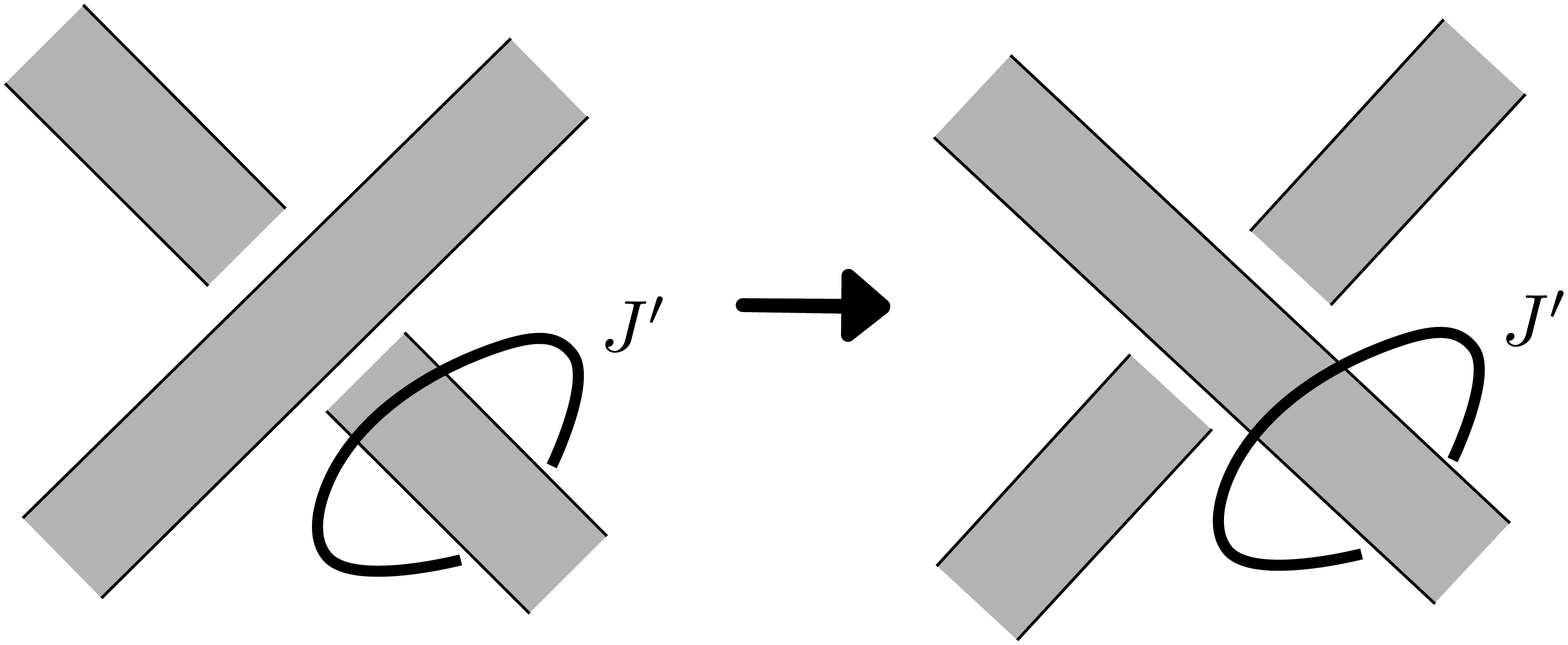}
\caption{Passing pieces of the band and the knot $K'$ over the disk bounded by $J'$ (not pictured).}
\label{fig:unknotting}
\end{center}
\end{subfigure}
\caption{}
\end{center}
\end{figure}

Then by the resolution of the Poenaru Conjecture \cite{pap:gabaifoliations3}, $J$ is an unknot in $S^3$.  Then for any disk $D$ in $S^3$ with $J = \partial D$, we must have  $K\cap D \neq \emptyset$, as otherwise we could perform in $S^3_r(K)$ the same debanding as we did in $Y$  to see that $K'$ is a nontrivial connected sum, and this would contradict the fact that $K$ is hyperbolic.  Thus no disk bounded by $J$ is disjoint from $K$, so $K$ lies in the solid torus $V = \overline{S^3 \setminus N(J)}$, and does not lie inside a ball in $V$.  Further, since $J$ was an unknot in $S^3_r(K)$ by construction, the manifold $V_r(K)$ resulting from surgery on $K$ in $V$ is reducible.  Hence $K$ is a cable by \cite{pap:scharlemann-reducible}, contradicting the assumption that $K$ is hyperbolic.
\end{proof}

\begin{proof}[Proof of Corollary \ref{cor:primeweb}]
Suppose $\Lambda$ is a great web in the general case with $l=v$ (if $v$ is prime this is forced by Proposition \ref{prop:webdivisibilty}), so that $\Lambda$ consists of a Scharlemann cycle $\sigma$ on the labels $(1,2)$ and a collection of edges which connect two vertices in $\sigma$.  Note then that there is only one Scharlemann cycle in $\Lambda$.  Suppose every edge of $\Lambda$ is parallel to an edge of $\sigma$.  Pick any vertex $v_0$ of $\sigma$ not containing the $\lambda$-ghost of $\Lambda$, for $\lambda = \frac{p+2}{2}$.  Then the $\lambda$-edge of $v_0$ is parallel to one of the two edges of $\sigma$ incident to $v_0$, and either way gives rise to a full quota.  Thus suppose not all edges of $\Lambda$ are parallel to edges of $\sigma$, and let $e$ be such an edge.  Then there is an arc $\gamma$ which is a subarc of $\partial\sigma$ with $\partial\gamma=\partial e$, such that $\gamma \cup e$ bounds a disk $D$ in $\Lambda$.  By the definition of $e$, $\gamma$ contains an interior vertex. By considering edges $e'$ interior to $D$ which are not parallel to edges of $\gamma$, one finds that there must be some interior vertex $v_0$ of $\gamma$ which is incident only to the two vertices adjacent to it along $\gamma$.  Such a $v_0$ necessarily gives rise to a full quota.  

In either situation, we see that $\Lambda$ is small, contradicting Proposition \ref{prop:largewebs}.
\end{proof}

Finally, we conclude the paper by showing a strengthening of \cite[Theorem 4.6]{pap:howiethreesummands} under the more restrictive setting of being contained in a great web.  

\begin{proof}[Proof of Corollary \ref{cor:noPSD}]
Suppose that $\sigma_1\cap\sigma_2$ contains at least $v_a\cup v_b$, for vertices $v_a$ and $v_b$ of $\Lambda$.  Note that we know that the two Scharlemann cycles are different lengths, by Lemma \ref{lem:howiescsarenice}(3).  By Proposition \ref{prop:webdivisibilty}, we see that $v=nl_1l_2$.  Then we see that the number $n_1$ of $\Gamma$-edges incident to $u_1$ lying interior to a single Scharlemann region of $G_P^1$ corresponding to $\sigma_1$ is the same for each region.  Thus $n_1=\frac{nl_1l_2-l_1}{l_1} = nl_2-1$.  Likewise, the number $n_x$ of $\Gamma$-edges incident to $u_x$ lying interior to a single Scharlemann region of $G_P^2$ corresponding to $\sigma_2$ has $n_x=nl_1-1$.  See Figure \ref{fig:biggamma}. Now as labels, $|a-b|$ is fixed, and this quantity can be counted by using the number of edge endpoints between the labels at $u_1$, and also at $u_x$, and we may set these two counts equal to one another.  However, we are not interested in $|a-b|$; we will instead count just those labels of $\Gamma$-edges lying between $a$ and $b$.  We see that, if between $a$ and $b$ on $G_P^i$ we have $k_i$ edges of $\sigma_i$, then
\begin{align*}
k_1 + (k_1+1)\cdot(nl_2-1) &= k_2 + (k_2+1)\cdot(nl_1-1),\\
k_1 + nl_2 + k_1nl_2-k_1 -1 &= k_2 + nl_1 +k_2nl_1-k_2-1,\\
k_1l_2 &= k_2l_1.
\end{align*}
Since $l_1$ and $l_2$ are relatively prime, we see that $l_2$ divides $k_2$.  But $k_2$ is some number of Scharlemann cycle edges of $\sigma_2$ lying in $G_P^2$ between a particular pair of such Scharlemann cycle edges, and hence $k_2 \leq l_2 -2$ (\textit{i.e.} $k_2$ certainly doesn't count the $a$-edge nor the $b$-edge of $u_x$).  Hence $l_2$ could not possibly divide $k_2$, a contradiction.  Thus a pair of Scharlemann cycles in $\Lambda$ can share at most one vertex.
\end{proof}

\begin{figure}[hbt]
\begin{center}
\includegraphics[width = .75\textwidth]{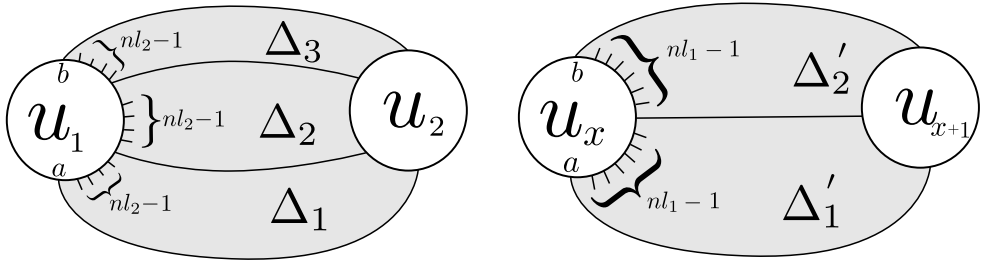}
\caption{The $\Gamma$-edges between $a$ and $b$ are counted for Corollary \ref{cor:noPSD}.  In the Figure, $k_1=2$ and $k_2=1$.}
\label{fig:biggamma}
\end{center}
\end{figure}